\documentclass[american,a4paper]{amsart}
\pdfoutput=1
\usepackage[latin9]{inputenc}
\usepackage{refstyle}
\usepackage{mathrsfs}
\usepackage{enumitem}
\usepackage{mathtools}
\usepackage{amssymb}
\usepackage{graphicx}
\usepackage[unicode=true,pdfusetitle,
 bookmarks=true,bookmarksnumbered=false,bookmarksopen=false,
 breaklinks=false,pdfborder={0 0 1},backref=false,colorlinks=false]{hyperref}
\usepackage{breakurl}
\usepackage{subcaption}
\usepackage{pgfplots}
\usepackage[foot]{amsaddr}

\providecommand{\propositionname}{Proposition}
\providecommand{\theoremname}{Theorem}
\providecommand{\definitionname}{Definition}

\numberwithin{equation}{section}
\numberwithin{figure}{section}
\theoremstyle{plain}
\newtheorem{thm}{\protect\theoremname}[section]
\theoremstyle{plain}
\newtheorem{prop}[thm]{\protect\propositionname}
\theoremstyle{definition}

\newref{prop}{name=Proposition~}
\newref{thm}{name=Theorem~}
\newref{lem}{name=Lemma~}
\newref{cor}{name=Corollary~}
\newref{conj}{name=Conjecture~}
\newref{fig}{name=Figure~}
\newref{subsec}{name=Subsection~}
\newref{sec}{name=Section~}
\newref{app}{name=Appendix~}
\newref{defn}{name=Definition~}

\DeclareMathOperator{\vspan}{span}
\DeclareMathOperator{\diag}{diag}
\DeclareMathOperator{\Real}{Re}

\begin{document}

\title[Oscillatory KPP system with three components]
{Two components is too simple: An example of oscillatory Fisher--KPP system with three components}

\author{Léo Girardin}

\thanks{This work was supported by a public grant as part of the Investissement d'avenir project,
reference ANR-11-LABX-0056-LMH, LabEx LMH. This work has been carried out in the framework of 
the NONLOCAL project (ANR-14-CE25-0013) funded by the French National Research Agency (ANR)}
\address{Laboratoire de Mathématiques d'Orsay, Université Paris
Sud, CNRS, Université Paris-Saclay, 91405 Orsay Cedex, France}
\email{leo.girardin@math.u-psud.fr}

\begin{abstract}
In a recent paper by Cantrell, Cosner and Yu \cite{Cantrell_Cosner_Yu_2018},
two-component KPP systems with competition of Lotka--Volterra type
were analyzed and their long-time behavior largely settled. In particular,
the authors established that any constant positive steady state,
if unique, is necessarily globally attractive. In the present paper,
we give an explicit and biologically very natural example of oscillatory three-component
system. Using elementary techniques or pre-established theorems, we show that it has a unique 
constant positive steady state with two-dimensional unstable 
manifold, a stable limit cycle, a predator--prey structure near the steady state,
periodic wave trains and point-to-periodic rapid traveling waves. 
Numerically, we also show the existence of pulsating fronts and propagating terraces.
\end{abstract}

\maketitle
\section{Introduction}

This paper is concerned with the long-time properties of the following
reaction--diffusion system:
\begin{equation}
    \tag{$\textup{KPP}_\mu$}
\left\{ \begin{matrix}\partial_{t}u_{1}-\Delta u_{1}= & u_{1}+\mu\left(-2u_{1}+u_{2}+u_{3}\right)-\frac{1}{10}\left(u_{1}+8u_{2}+u_{3}\right)u_{1}\,\\
\partial_{t}u_{2}-\Delta u_{2}= & u_{2}+\mu\left(+u_{1}-2u_{2}+u_{3}\right)-\frac{1}{10}\left(u_{1}+u_{2}+8u_{3}\right)u_{2}\,\\
\partial_{t}u_{3}-\Delta u_{3}= & u_{3}+\mu\left(+u_{1}+u_{2}-2u_{3}\right)-\frac{1}{10}\left(8u_{1}+u_{2}+u_{3}\right)u_{3},
\end{matrix}\right.\label{sys:1}
\end{equation}
written in vector form as:
\[
\partial_{t}\mathbf{u}-\Delta\mathbf{u}=\mathbf{u}+\mu\mathbf{M}\mathbf{u}-\left(\mathbf{C}\mathbf{u}\right)\circ\mathbf{u},
\]
 where $\mu$ is a positive constant,
\[
\mathbf{M}=\left(\begin{matrix}-2 & 1 & 1\\
1 & -2 & 1\\
1 & 1 & -2
\end{matrix}\right),\quad\mathbf{C}=\frac{1}{10}\left(\begin{matrix}1 & 8 & 1\\
1 & 1 & 8\\
8 & 1 & 1
\end{matrix}\right),
\]
and $\circ$ denotes the Hadamard product (component-by-component
product) of two vectors. 

This system (\ref{sys:1}) is a particular case of KPP system, as
defined by the author in \cite{Girardin_2016_2} and subsequently
analyzed in \cite{Girardin_2017}. Let us remind briefly here that
this name comes from the fact that the above reaction term is strongly
reminiscent of the scalar logistic term $u\left(1-u\right)$ and leads
to very similar conclusions regarding extinction, persistence, traveling
waves and spreading speed.

From the biological point of view, (\ref{sys:1}) can for instance
model a structured population with three coexisting phenotypes subjected
to spatial dispersal, phenotypical changes and competition for resources.
As explained by Cantrell, Cosner and Yu \cite{Cantrell_Cosner_Yu_2018},
the phenotypical changes can come from behavioral switching, phenotypic
plasticity or Darwinian evolution, for instance. 

The paper is organized as follows. In the remaining of this introductive section, 
we present and comment our results. In Section 2, we prove our main analytical result. In
Section 3, we present our numerical findings. In the appendix, we give elementary
proofs of related new results on periodic wave trains (\appref{WT}) or KPP systems (\appref{ES} and \appref{USS}).

\subsection{Main result}

In their paper, Cantrell, Cosner and Yu \cite{Cantrell_Cosner_Yu_2018}
studied the KPP system with competition of Lotka--Volterra type
\[
    \partial_{t}\mathbf{u}-\diag\left(\mathbf{d}\right)\Delta\mathbf{u}=\mathbf{L}\mathbf{u}-\mathbf{C}\mathbf{u}\circ\mathbf{u}
\]
with only two components but in full generality with respect to the parameters.
Here we recall that the minimal KPP assumptions are the positivity \footnote{In the whole paper, positive vectors are vectors with
positive components and nonnegative, negative and nonpositive vectors are defined analogously.} 
of $\mathbf{d}$ and $\mathbf{C}$ as well as the essential nonnegativity \footnote{Off-diagonal nonnegativity.} of $\mathbf{L}$, 
its irreducibility and the positivity of its Perron--Frobenius eigenvalue. 
We also recall that a competition term $\mathbf{c}\left( \mathbf{u} \right)\circ\mathbf{u}$
is referred to as a Lotka--Volterra competition term if the vector field $\mathbf{c}$
is linear, namely $\mathbf{c}\left( \mathbf{u} \right)=\mathbf{C}\mathbf{u}$.

Cantrell, Cosner and Yu obtained an almost complete characterization of the
long-time asymptotics in bounded domains. In particular, they proved 
that any constant positive steady state, if unique, is globally
attractive under Neumann boundary conditions. Hence it is an important
step forward regarding the general study of two-component KPP systems
with Lotka--Volterra competition, which have been studied
by several authors in the past few years 
\cite{Elliott_Cornel,Girardin_2017,Griette_Raoul,Morris_Borger_Crooks}. 
We will show briefly in \propref{2_compo_convergence} how their 
results and arguments of proof can be applied to the unbounded setting.

On the contrary, for the three-component system (\ref{sys:1}), a qualitatively completely
different result will be proved in the forthcoming pages. Before giving the statement, 
we point out that the parameters of the system are obviously normalized in such a way 
that, for any value of $\mu$, $\mathbf{1}=\left(1,1,1\right)^{T}$ is a positive constant steady state.
Additionally, we define $\mu_{\textup{H}}=\frac{7}{60}$,
$\mu_{-}=\frac{1}{10}$ and $\mu_{+}=\frac{8}{10}$. These values satisfy
\[
    0<\mu_{-}<\mu_{\textup{H}}<\mu_{+}<1.
\]
In the statement below, $\left( \mathbf{e}_i \right)_{i\in\left\{ 1,2,3 \right\}}$ denotes the canonical 
basis of $\mathbb{R}^3$. For notation convenience, indices are defined modulo 
$3$ (\textit{i.e.} $\mathbf{e}_4=\mathbf{e}_1$, $\mathbf{e}_5=\mathbf{e}_2$, etc.).

\begin{thm}
\label{thm:1}
The diffusionless system
\begin{equation}
    \tag{$\textup{KPP}^0_{\mu}$}
\dot{\mathbf{u}}=\mathbf{u}+\mu\mathbf{M}\mathbf{u}-\left(\mathbf{C}\mathbf{u}\right)\circ\mathbf{u}
\label{sys:2}
\end{equation}
satisfies the following properties.
\begin{enumerate}
    \item $\mathbf{1}$ is the unique positive steady state. 
\item If $\mu>\mu_{\textup{H}}$, $\mathbf{1}$ is locally asymptotically stable,
    but at $\mu=\mu_{\textup{H}}$, it undergoes a supercritical Hopf bifurcation
leading to the birth of a unique and locally asymptotically stable
limit cycle $\mathsf{C}_{\mu}$. Using $\mu$ as a parameter, there exists a family of positive 
limit cycles
$\left( \mathsf{C}_\mu \right)_{\mu\in\left( 0,\mu_{\textup{H}} \right)}$
and any such family converges, in the sense of the Hausdorff distance, as $\mu\to 0$, to
\[
    \mathsf{C}_0 =\bigcup_{i\in\{1,2,3\}}\{10\mathbf{e}_i\}\cup\mathsf{H}_i,
\]
where $\mathsf{H}_i$ is, for (\ref{sys:2}) with $\mu=0$, the unique heteroclinic connection 
between $10\mathbf{e}_i$ and $10\mathbf{e}_{i+1}$, which lies in $\mathbf{e}_{i+2}^{\perp}$.
Furthermore, any limit cycle $\mathsf{C}_\mu$ satisfies
\[
    \mathsf{C}_\mu\subset\left\{ \mathbf{v}\in\left( 0,+\infty \right)^3\ |\ 1\leq\frac{v_1+v_2+v_3}{3}\leq\frac{10}{3} \right\}
\]
and is rotating clockwise around $\vspan\left( \mathbf{1} \right)$ if seen from $\mathbf{0}$.
\item Let $\mathbf{v}\geq\mathbf{0}$. The reaction term at $\mathbf{v}$ is cooperative
    if and only if 
    \[
	\mathbf{v}\in\left[ 0,\frac{\mu}{\mu_+} \right]^3
    \]
    and competitive if and only if 
    \[
	\mathbf{v}\in\left[ \frac{\mu}{\mu_-},+\infty\right)^3.
    \]
    If $\mathbf{v}\in\left( \frac{\mu}{\mu_+},\frac{\mu}{\mu_-} \right)^3$, the off-diagonal
    entries of the linearized reaction term at $\mathbf{v}$ have the following signs:
	\[
	\left(\begin{matrix}\bullet & - & +\\
	+ & \bullet & -\\
	- & + & \bullet
	\end{matrix}\right).
	\]
	This is in particular the case if $\mathbf{v}=\mathbf{1}$ and 
	$\mu\in\left( \mu_-,\mu_+ \right)$.
\end{enumerate}
\end{thm}

We will illustrate numerically that the limit cycle $\mathsf{C}_{\mu}$
seems to be in fact globally attractive (with respect to initial conditions
that are not in the basin of attraction of $\mathbf{0}$ or $\mathbf{1}$,
namely almost all of them), and therefore also unique. However,
we did not manage to prove the global attractivity or the uniqueness.

\subsection{Discussion on \thmref{1}}
Although the third property follows from a direct differentiation, it is qualitatively
very meaningful. On one hand, in the cube $\left( \frac{10\mu}{8},10\mu \right)^3$, the system has
the structure of a cyclic predator--prey system (rock--paper--scissor-like).
On the other hand, a consequence of Cantrell--Cosner--Yu 
\cite[Propositions 2.5 and 3.1]{Cantrell_Cosner_Yu_2018}
is that two-component KPP systems with Lotka--Volterra competition are
competitive in the neighborhood of any constant positive saddle.
The case $\mu\in\left(\mu_-,\mu_{\textup{H}}\right)$ of \thmref{1} above
proves that this property fails with three components.
For the sake of completeness, we will prove in \propref{unstable_steady_state} 
what seems to be the optimal result for an arbitrary number of components: 
at an unstable constant positive steady state, the reaction term is not cooperative. 

Notice that, changing in (\ref{sys:2}) $\mu$ into $\frac{1}{\mu}$
and normalizing appropriately the time variable, we obtain the system 
\[
    \partial_{t}\mathbf{u}-\mathbf{M}\mathbf{u}=\mu\left(\mathbf{u}-\left(\mathbf{C}\mathbf{u}\right)\circ\mathbf{u}\right),
\]
which is obviously similar (by finite difference approximations and Riemann sum approximations, see \cite[Section 1.5]{Girardin_2016_2}) 
to the following nonlocal KPP equation:
\[
\partial_{t}u-\partial_{yy}u=\mu u\left(1-\phi\star_{y}u\right).
\]
For this equation, it is now well-known that as $\mu$ increases, the steady state
$1$ is dynamically destabilized (we refer for instance to Berestycki--Nadin--Perthame--Ryzhik 
\cite{Berestycki_Nadin_Perthame_Ryzhik}, Faye--Holzer \cite{Faye_Holzer_14} and 
Nadin--Perthame--Tang \cite{Nadin_Perthame_Tang}). 
In this context, this property is usually understood as a form of Turing instability.
In particular, Nadin--Perthame--Tang \cite{Nadin_Perthame_Tang} showed numerically how this Turing
instability can lead to interesting spreading phenomena, where the classical traveling waves
connecting $0$ to $1$ are replaced by more sophisticated solutions.

\subsection{Discussion and numerical results on the spatial structure}
Of course, the presence in (\ref{sys:1}) of a third variable $x$
makes the system (\ref{sys:1}) qualitatively different from the nonlocal KPP equation. In fact,
as explained by the author in \cite{Girardin_2016_2}, 
(\ref{sys:1}) is more reminiscent of the cane toad equation with nonlocal competition and
local or nonlocal mutations 
\cite{Alfaro_Coville_Raoul,Arnold_Desvill,Benichou_Calvez,Bouin_Calvez_2014,Bouin_Calvez_2,Bouin_Henderso,PrevostPhD}.
For these equations, what happens in the wake of an invasion front is still poorly understood. 
Our results and numerical findings are, in this regard, quite interesting. 

Taking profit of the Hopf bifurcation at $\mu=\mu_{\textup{H}}$, we can apply 
a theorem of Kopell--Howard \cite{Kopell_Howard_1973} (we also refer to Murray 
\cite[Chapter 1, Section 1.7]{Murray_II}) and immediately obtain the following
result.

\begin{prop}\label{prop:wave_train}
    Assume that (\ref{sys:1}) is set in the spatial domain $\mathbb{R}^n$ with $n\in\mathbb{N}$. 

    If $\mu<\mu_{\textup{H}}$ and $\mathsf{C}_\mu$ is locally asymptotically stable,
    then for any $e\in\mathbb{S}^{n-1}$,
(\ref{sys:1}) admits a continuous one-parameter
family of traveling plane wave train solutions of the form 
\[
    \mathbf{u}:\left( t,x \right)\mapsto \mathbf{p}_{\gamma}\left( \kappa_\gamma x\cdot e-\sigma_{\gamma}t \right)
\]
where $\gamma\in\left[ 0,\gamma_{\textup{s}} \right]\cup\left[ \gamma_{\textup{l}},1 \right]$ is the parameter, 
$0<\gamma_{\textup{s}}\leq\gamma_{\textup{l}}<1$,
$\kappa_\gamma\in\mathbb{R}$, $\sigma_\gamma\in\mathbb{R}$ and $\mathbf{p}_\gamma$ is 
positive and periodic. 
Without loss of generality, $\gamma$ can be understood as an amplitude parameter:
\begin{itemize}
    \item the image of $\mathbf{p}_0$ is $\mathbf{1}$;
    \item $\kappa_1=0$ and the image and period $\sigma_1^{-1}$ of $\mathbf{p}_1$ are respectively
	the limit cycle $\mathsf{C}_\mu$ and its associated period;
    \item $\gamma\mapsto\mathbf{p}_\gamma$ is increasing in the sense that the image of 
	\[
	    \left(\gamma,\xi\right)\in\left[ 0,\gamma_1 \right]\times\mathbb{R}\mapsto\mathbf{p}_{\gamma}\left( \xi \right)
	\]
	is strictly included in that of 
	\[
	    \left( \gamma,\xi \right)\in\left[ 0,\gamma_2 \right]\times\mathbb{R}\mapsto\mathbf{p}_{\gamma}\left( \xi \right)
	\]
	provided $\gamma_1<\gamma_2$.
\end{itemize}

Furthermore, there exists $\overline{\gamma}\in\left(0,\gamma_{\textup{s}}\right]$ such that 
all wave trains of 
amplitude $\gamma\in\left[0,\overline{\gamma}\right)$ are unstable with respect to 
compactly supported, bounded perturbations.
\end{prop}

Beware that the Kopell--Howard theorem only shows that there are wave trains close to 
$\mathbf{1}$ (small amplitude $0\leq\gamma\leq\gamma_{\textup{s}}$) and close to 
$\mathsf{C}_\mu$ (large amplitude $\gamma_{\textup{l}}\leq\gamma\leq 1$).
It is unclear, even numerically, whether a continuum of wave trains exists 
(\textit{i.e.}, the equality $\gamma_{\textup{s}}=\gamma_{\textup{l}}$ is unclear).

The nonlinear stability of a wave train of amplitude close to $1$ is a delicate question, as 
established by Kopell--Howard \cite{Kopell_Howard_1973} and subsequently confirmed by
Maginu \cite{Maginu_1978,Maginu_1981}. Nevertheless, simply thanks to the fact that the 
diffusion matrix is the identity, the stability of the limit cycle extends in the following way.

\begin{prop}\label{prop:stability_wave_trains}
    If $\mu<\mu_{\textup{H}}$ and $\mathsf{C}_\mu$ is locally asymptotically stable with
    respect to (\ref{sys:2}), then 
    there exists $\underline{\gamma}\in\left[\gamma_{\textup{l}},1\right)$ such that all
    wave trains of amplitude $\gamma\in\left(\underline{\gamma},1\right]$ are 
    marginally stable in linearized criterion.
\end{prop}

The proof of \propref{stability_wave_trains} is very simple but is actually 
not provided in \cite{Maginu_1981}. For the sake of completeness, it will be detailed in \appref{WT}.
Note that the notion of stability in the statement above is the marginal stability in 
linearized criterion \cite{Fife_1979} and not the asymptotic waveform stability \cite{Maginu_1981},
which might fail in general and remains a difficult and open question.

Numerically, we will observe \textit{propagating terraces} (succession of compatible
waves with decreasingly ordered speeds, as defined by Ducrot,
Giletti and Matano \cite{Ducrot_Giletti_Matano}) where $\mathbf{0}$
is invaded by $\mathbf{1}$ and then $\mathbf{1}$ is slowly invaded
by a stable wave train of amplitude $\gamma$ close to $1$. 
The former invasion takes the form of a \textit{traveling wave} (here defined as an 
entire solution with constant profile and speed) whereas the latter takes the 
form of a \textit{pulsating front} (more general entire solution connecting at some constant speed
two periodic, possibly homogeneous, solutions, defined for instance by Nadin in \cite{Nadin_2009} and also known 
as pulsating traveling wave). 

Regarding traveling waves, the following proposition can
be straightforwardly established by looking for wave profiles of the
form $\xi\mapsto p\left(\xi\right)\mathbf{1}$ (we refer to \cite{Girardin_2017}
for a similar construction).
\begin{prop}
\label{prop:2}
The system (\ref{sys:1}) admits a family of monotonic traveling plane
waves connecting $\mathbf{0}$ to $\mathbf{1}$ at speed $c\geq2$. 
\end{prop}
This is indeed such a monotonic profile we observe numerically.

By direct application of theorems due to Fraile and Sabina 
\cite{Fraile_Sabina_1985,Fraile_Sabina_1989}, there exist also point-to-periodic rapid 
traveling waves connecting $\mathbf{1}$ to wave trains of large amplitude;
however, these are not the pulsating fronts that we observe numerically, which do not
have a constant profile and have two distinct speeds, the one of the invasion front
and the one of the wave train. 

More precisely, close to the bifurcation value $\mu_{\textup{H}}$,
the speed of the pulsating front of the terrace connecting $\mathbf{1}$ to the wave train is 
$2\sqrt{3\left( \mu_{\textup{H}}-\mu \right)}$. On the contrary, the intrinsic speed of the wave
train, $c_\gamma = \frac{\sigma_\gamma}{\kappa_\gamma}$, is negative and of very large absolute 
value (consistently with $\left|c_\gamma\right|\to\infty$ as $\gamma\to 1$). 
We emphasize that the preceding formula for the invasion speed is linearly determinate (in some
sense precised below in \subsecref{numerics_diffusion}) and was first predicted heuristically 
by Sherratt \cite{Sherratt_1998}.
Interestingly, in Sherratt's predictions, both the invasion speed and the wave train speed
do not depend on the initial condition or even on the speed of the first invasion 
($\mathbf{1}$ into $\mathbf{0}$). This is confirmed by our numerical experiments.

Since this pulsating front is parametrized by two distinct speeds (that of the invasion
and that of the wave train), the interesting problem of its existence is in fact very difficult. 
Seemingly similar
results on periodic wave trains \cite{Kopell_Howard_1973,Murray_II}, point-to-periodic rapid 
traveling waves \cite{Fraile_Sabina_1985,Fraile_Sabina_1989} or even space-periodic pulsating fronts 
(recently studied by Faye and Holzer \cite{Faye_Holzer_14}) are proved by means of codimension 
$1$ bifurcation arguments. The space-time periodic pulsating front at hand is a codimension 
$2$ bifurcation problem. Its resolution is definitely outside the scope of this paper and 
we leave it for future work. 

Another prediction of Sherratt \cite{Sherratt_1998} is the possible
nonexistence of such propagating terraces when the speed $c_1$ of the first
invasion satisfies 
\[
    c_1>\frac{\frac{7\sqrt{3}}{20}}{\sqrt{3\left( \mu_{\textup{H}}-\mu \right)}}=
    \frac{7}{20\sqrt{\mu_{\textup{H}}-\mu}}.
\]
More precisely, when this condition holds, a periodic wave train of speed $c_1$ 
and small amplitude exists and therefore there is the possibility of
a point-to-periodic rapid traveling wave connecting directly $\mathbf{0}$ to this unstable 
wave train. 
Recall however that solutions that are initially compactly supported asymptotically spread
at speed $2$ \cite{Girardin_2016_2}, which is clearly smaller than the above
threshold close to the bifurcation value,
whence these traveling waves are irrelevant regarding biological applications.
Anyway, even with initial conditions with appropriate exponential decay \cite{Girardin_2017},
we numerically obtained propagating terraces and did not manage to catch 
these non-monotonic rapid traveling waves. Thus
their existence remains a completely open problem. We point out here that this existence 
would be in sharp contrast with a nonexistence result by Alfaro and Coville for the nonlocal
Fisher--KPP equation \cite{Alfaro_Coville_2012}. 

Thanks to the $\lambda$-$\omega$ normal form, Sherratt manages also to find formulas for
the amplitude $\gamma$ and the speed $c_\gamma$ of the wave train. Nevertheless,
it is quite tedious to reduce our three-component system to the appropriate 
two-component $\lambda$-$\omega$ system, as its phase ``plane'' is the unstable manifold 
of the steady state $\mathbf{1}$ (which is definitely not a Euclidean plane).
For the sake of brevity, we choose to omit here this reduction and the precise predictions
on the wave train. 

The contrast with the very simple dynamics exhibited by Cantrell,
Cosner and Yu for the two-component system is striking. This is of
course reminiscent of the contrast between two-component and three-component
competitive Lotka--Volterra systems: the two-component ones always
have a simple monostable or bistable structure, devoid of periodic
orbits, whereas some three-component ones have stable limit cycles
(as established by Zeeman \cite{Zeeman_1993} in her classification
of the 33 stable nullcline equivalence classes). However let us emphasize once more
that our system is not competitive near $\mathbf{1}$, so that
qualitatively similar observations for three-species competitive systems
(for instance, those of Petrovskii--Kawasaki--Takasu--Shigesada 
\cite{Petrovskii_Kawasaki_Takasu_Shigesada}) are actually
unrelated to our results.

We emphasize that here all diffusion rates are equal (and normalized), whence there is
no Turing instability with respect to the space variable. Obviously, if the phenotypes 
differ also in diffusion rate, then even more complicated dynamics are to be expected -- and
can be observed numerically. On this vast topic, we refer for instance to 
Smith--Sherratt \cite{Smith_Sherratt}.

This collection of results confirms that the traveling
waves constructed in \cite{Girardin_2016_2} are definitely not the end of the story from
the viewpoint of the asymptotic spreading for the Cauchy problem. 

\subsection{What about more general systems?}
\thmref{1},1-2 (and its various consequences) can be easily extended to general KPP systems with
Lotka--Volterra competition, equal diffusion rates and any number of components 
provided the matrices $\mathbf{L}$ and $\mathbf{C}$ remain circulant matrices and a Hopf
bifurcation does occur (thus some asymmetry is required). \thmref{1},3 needs a bit more 
care and appropriately chosen coefficients but should still hold true in a much more 
general framework. 

Here, we choose to focus on a particular three-dimensional example, mainly because
our point is to confirm the existence of oscillatory KPP systems. There are secondary reasons
worth mentioning: first, the particular choice we make simplifies a lot the notations and 
calculations; second, Hopf bifurcations with a hyperbolic transverse component are at their 
core a three-dimensional phenomenon and taking into account more dimensions is just cumbersome. 

\section{Proof of \thmref{1}}

\subsection{Well-known facts on circulant matrices}

The matrices $\mathbf{I}$, $\mathbf{M}$ and $\mathbf{C}$ are all
circulant matrices. Recall that the $3\times3$ circulant matrix 
\[
\left(\begin{matrix}a & b & c\\
c & a & b\\
b & c & a
\end{matrix}\right)
\]
 admits as eigenpairs $\left(a+b+c,\mathbf{1}\right)$, $\left(a+b\textup{j}+c\overline{\textup{j}},\mathbf{z}\right)$
and $\left(a+b\overline{\textup{j}}+c\textup{j},\overline{\mathbf{z}}\right)$,
where 
\[
    \textup{j}=\exp\left(\frac{2\textup{i}\pi}{3}\right)\textup{ and }
\mathbf{z}=\frac{1}{\sqrt{3}}\left(\begin{matrix}1\\
\textup{j}\\
\overline{\textup{j}}
\end{matrix}\right).
\]
Recall that $\textup{j}+\overline{\textup{j}}=-1$,
$\mathbf{z}\circ\mathbf{z}=\frac{1}{\sqrt{3}}\overline{\mathbf{z}}$ and
$\mathbf{z}\circ\overline{\mathbf{z}}=\frac{1}{3}\mathbf{1}$.

Recall also that the set of all $n\times n$ circulant matrices forms
a commutative algebra and that the matrix 
\[
\mathbf{U}=\frac{1}{\sqrt{3}}\left(\begin{matrix}1 & 1 & 1\\
1 & \textup{j} & \overline{\textup{j}}\\
1 & \overline{\textup{j}} & \textup{j}
\end{matrix}\right)
\]
 is a unitary matrix such that 
\[
\left(\begin{matrix}a & b & c\\
c & a & b\\
b & c & a
\end{matrix}\right)=\mathbf{U}\left(\begin{matrix}a+b+c & 0 & 0\\
0 & a+b\textup{j}+c\overline{\textup{j}} & 0\\
0 & 0 & a+b\overline{\textup{j}}+c\textup{j}
\end{matrix}\right)\overline{\mathbf{U}}^{T}.
\]

These basic properties bring forth a very convenient spectral decomposition
for the problem. 

\subsection{Uniqueness of the constant positive steady state}
\begin{proof}
    Let $\mu>0$ and $\mathbf{v}\geq\mathbf{0}$
be a solution of $\left(\mathbf{I}+\mu\mathbf{M}\right)\mathbf{v}-\mathbf{C}\mathbf{v}\circ\mathbf{v}=\mathbf{0}$. 

Writing $\mathbf{v}=\alpha\mathbf{1}+\beta\mathbf{z}+\gamma\overline{\mathbf{z}}$
with $\alpha,\beta,\gamma\in\mathbb{C}$ and identifying the real and imaginary parts, we 
straightforwardly verify that $\alpha\in\left[0,+\infty\right)$ (by nonnegativity of $\mathbf{v}$) 
and $\beta=\overline{\gamma}$ (by the fact that $\mathbf{v}$ is a real vector). 

Then the equality $\left(\mathbf{I}+\mu\mathbf{M}\right)\mathbf{v}=\mathbf{C}\mathbf{v}\circ\mathbf{v}$
reads 
\[
\alpha\mathbf{1}+\left(1-3\mu\right)\beta\mathbf{z}+\left(1-3\mu\right)\overline{\beta\mathbf{z}} =\left(\alpha\mathbf{1}+\frac{7\textup{j}}{10}\beta\mathbf{z}+\frac{7\overline{\textup{j}}}{10}\overline{\beta}\overline{\mathbf{z}}\right)\circ\left(\alpha\mathbf{1}+\beta\mathbf{z}+\overline{\beta}\overline{\mathbf{z}}\right).
\]
After a few algebraic manipulations, this is equivalent to the system
\[
\left\{ \begin{matrix}\alpha^{2}-\alpha-\frac{7}{30}\left|\beta\right|^{2} & =0\,\\
    \frac{7\sqrt{3}}{30}\overline{\textup{j}\beta^{2}}+\left(\alpha-\left(1-3\mu\right)+\frac{7}{10}\left(\frac{-1+\textup{i}\sqrt{3}}{2}\right)\alpha\right)\beta & =0.
\end{matrix}\right.
\]


On one hand, assuming by contradiction the existence of a solution such that $\left|\beta\right|\neq 0$
and taking the square of the modulus of the second line multiplied by $20$, we deduce 
\[
\frac{196}{3}\left|\beta\right|^{2}=\left(13\alpha-20\left(1-3\mu\right)\right)^{2}+147\alpha^{2}=316\alpha^2-520\left( 1-3\mu \right)\alpha+400\left( 1-3\mu \right)^2,
\]
that is
\[
\frac{49}{3}\left|\beta\right|^{2}=79\alpha^2-130\left( 1-3\mu \right)\alpha+100\left( 1-3\mu \right)^2.
\]

On the other hand, from the first line, we deduce $\frac{49}{3}\left|\beta\right|^{2}=70\left(\alpha^2-\alpha\right)$.

Equalizing the two expressions of $\frac{49}{3}\left|\beta\right|^2$, we obtain 
\[
    9\alpha^2+10\left( 7-13\left( 1-3\mu \right) \right)\alpha+100\left( 1-3\mu \right)^2=0.
\]

On one hand, the discriminant of this equation is
$100\left(\left(7-13\left(1-3\mu\right)\right)^{2}-36\left(1-3\mu\right)^{2}\right)$, which is itself nonnegative if and only if 
\[
7-26\left(1-3\mu\right)+19\left(1-3\mu\right)^{2}\geq0,
\]
that is if and only if $1-3\mu\notin\left(\frac{7}{19},1\right)$,
that is if and only if $\mu\notin\left(0,\frac{4}{19}\right)$.
Therefore, $\alpha$ being real, $\mathbf{v}$ cannot possibly exist if $\mu\in\left(0,\frac{4}{19}\right)$.
Hence necessarily $\mu\geq\frac{4}{19}$.

On the other hand, at $\tilde{\alpha}=0$, the polynomial 
$\tilde{\alpha}\mapsto 9\tilde{\alpha}^2+10\left( 7-13\left( 1-3\mu \right) \right)\tilde{\alpha}+100\left( 1-3\mu \right)^2$
is nonnegative and with derivative $30\left( 13\mu-2 \right)$, which is positive since we now assume
the necessary condition $\mu\geq\frac{4}{19}>\frac{2}{13}$.
Therefore the polynomial has actually no zero in $\left(0,+\infty\right)$.
Since $\mathbf{v}\geq \mathbf{0}$ and $\left|\beta\right|\neq 0$ imply together $\alpha>0$, we find
a contradiction.

This exactly means that all solutions satisfy $\beta=0$ and $\alpha^{2}-\alpha=0$, so that $\mathbf{0}$
and $\mathbf{1}$ are indeed the only solutions as soon as $\mu>0$.

\end{proof}

\subsection{The linearization at $\mathbf{1}$: eigenelements and Hopf
bifurcation}
\begin{proof}
The change of variable $\mathbf{v}=\mathbf{1}+\mathbf{w}$ leads to
\[
\left(\mathbf{I}+\mu\mathbf{M}\right)\mathbf{v}-\mathbf{C}\mathbf{v}\circ\mathbf{v}=\left(\mu\mathbf{M}-\mathbf{C}\right)\mathbf{w}-\mathbf{C}\mathbf{w}\circ\mathbf{w}.
\]
 Hence the linearization of the reaction term at $\mathbf{v}=\mathbf{1}$
is exactly 
\[
\mu\mathbf{M}-\mathbf{C}=\left(\begin{matrix}-2\mu-\frac{1}{10} & \mu-\frac{8}{10} & \mu-\frac{1}{10}\\
\mu-\frac{1}{10} & -2\mu-\frac{1}{10} & \mu-\frac{8}{10}\\
\mu-\frac{8}{10} & \mu-\frac{1}{10} & -2\mu-\frac{1}{10}
\end{matrix}\right).
\]
(This is of course consistent with a direct differentiation.)

Since $\mu\mathbf{M}-\mathbf{C}$ is a circulant matrix, three complex
eigenpairs are $\left(-1,\mathbf{1}\right)$, $\left(\lambda_{\mu},\mathbf{z}\right)$,
$\left(\overline{\lambda_{\mu}},\overline{\mathbf{z}}\right)$, where
\begin{align*}
\lambda_{\mu} & =-2\mu-\frac{1}{10}+\left(\mu-\frac{8}{10}\right)\textup{j}+\left(\mu-\frac{1}{10}\right)\overline{\textup{j}}\\
 & =3\left(\frac{7}{60}-\mu\right)+\textup{i}\frac{7\sqrt{3}}{20}.
\end{align*}

This proves indeed the local asymptotic stability when $\mu>\frac{7}{60}$,
the Hopf bifurcation at $\mu=\frac{7}{60}$ and, as the transverse
component is hyperbolic, the uniqueness of the limit cycle $\mathsf{C}_{\mu}$ close 
to the bifurcation value.
\end{proof}
Thereafter we will also need adjoint eigenvectors satisfying $\overline{\left(\mu\mathbf{M}-\mathbf{C}\right)^{T}}\mathbf{z}=\lambda\mathbf{z}$,
that is $\left(\mu\mathbf{M}-\mathbf{C}\right)^{T}\mathbf{z}=\lambda\mathbf{z}$.
Using this time the fact that $\left(\mu\mathbf{M}-\mathbf{C}\right)^{T}$
is circulant, eigenpairs of it are $\left(-1,\mathbf{1}\right)$,
$\left(\lambda_{\mu},\overline{\mathbf{z}}\right)$, $\left(\overline{\lambda_{\mu}},\mathbf{z}\right)$. 

\subsection{The first Lyapunov coefficient}

First, we recall a well-known statement (we refer, for instance, to Kuznetsov \cite[Formula 5.39, p. 180]{Kuznetsov_2004}). 

\begin{thm}
Let $N\in\mathbb{N}$ such that $N\geq2$, $I\subset\mathbb{R}$ be an interval containing
$0$ and $\mathbf{f}\in\mathscr{C}^{3}\left(\mathbb{R}^{N}\times\mathbb{R},\mathbb{R}^{N}\right)$
such that $\mathbf{f}\left(\mathbf{0},\eta\right)=\mathbf{0}$ for
all $\eta\in I$. 

Assume that, if $\eta\in I\cap\left(-\infty,0\right)$,
$\mathbf{0}$ is a locally asymptotically stable steady state for
the dynamical system 
\[
    \dot{\mathbf{u}}=\mathbf{f}\left(\mathbf{u},\eta\right),
\]
and that at $\eta=0$ it undergoes a Hopf bifurcation (with a center subspace of dimension $2$). 

Let $\mathbf{A}\in\mathsf{M}_{N}\left(\mathbb{R}\right)$, $\mathbf{b}:\mathbb{R}^{N}\times\mathbb{R}^{N}\to\mathbb{R}^{N}$
and $\mathbf{c}:\mathbb{R}^{N}\times\mathbb{R}^{N}\times\mathbb{R}^{N}\to\mathbb{R}^{N}$
such that the Taylor expansion of $\mathbf{u}\mapsto\mathbf{f}\left(\mathbf{u},0\right)$
at $\mathbf{0}$ has the form 
\[
\mathbf{f}\left(\mathbf{u},0\right)=\mathbf{A}\mathbf{u}+\frac{1}{2}\mathbf{b}\left(\mathbf{u},\mathbf{u}\right)+\frac{1}{6}\mathbf{c}\left(\mathbf{u},\mathbf{u},\mathbf{u}\right)+O\left(\left|\mathbf{u}\right|^{4}\right).
\]
 Let $\mathbf{q}\in\mathbb{C}^{N}$ be an eigenvector of $\mathbf{A}$
associated with the purely imaginary eigenvalue $\lambda\in\textup{i}\mathbb{R}_{+}$
and $\mathbf{p}\in\mathbb{C}^{N}$ be an eigenvector of $\mathbf{A}^{T}$
associated with $-\lambda$, normalized so that $\overline{\mathbf{p}}^{T}\mathbf{q}=\overline{\mathbf{q}}^{T}\mathbf{q}=1$.

Then the Hopf bifurcation is supercritical, respectively subcritical,
if the first Lyapunov coefficient 
\[
l_{1}\left(0\right)=\frac{1}{2\left|\lambda\right|}\Real\left[\overline{\mathbf{p}}^{T}\mathbf{c}\left(\mathbf{q},\mathbf{q},\overline{\mathbf{q}}\right)-2\overline{\mathbf{p}}^{T}\mathbf{b}\left(\mathbf{q},\mathbf{A}^{-1}\mathbf{B}\left(\mathbf{q},\overline{\mathbf{q}}\right)\right)+\overline{\mathbf{p}}^{T}\mathbf{c}\left(\overline{\mathbf{q}},\left(2\textup{i}\omega_{0}\mathbf{I}_{N}-\mathbf{A}\right)^{-1}\mathbf{B}\left(\mathbf{q},\mathbf{q}\right)\right)\right]
\]
 is negative, respectively positive. 
\end{thm}

We are now in position to apply this theorem to our case. 
\begin{proof}
Performing the changes of variable $\mathbf{v}=\mathbf{1}+\mathbf{w}$
and $\mu=\frac{7}{60}-\eta$ and identifying the Taylor expansion
of $\mathbf{w}\mapsto\left(\frac{7}{60}\mathbf{M}-\mathbf{C}\right)\mathbf{w}-\mathbf{C}\mathbf{w}\circ\mathbf{w}$
at $\mathbf{w}=\mathbf{0}$, which is actually an exact expansion,
we find
\[
\left(\frac{7}{60}\mathbf{M}-\mathbf{C}\right)\mathbf{w}-\mathbf{C}\mathbf{w}\circ\mathbf{w}=\mathbf{A}\mathbf{w}+\frac{1}{2}\mathbf{b}\left(\mathbf{w},\mathbf{w}\right),
\]
where $\mathbf{A}=\frac{7}{60}\mathbf{M}-\mathbf{C}$ and $\mathbf{b}:\left(\mathbf{v},\mathbf{w}\right)\mapsto-\mathbf{w}\circ\mathbf{C}\mathbf{v}-\mathbf{v}\circ\mathbf{C}\mathbf{w}$.
Let $\lambda=\lambda_{7/60}=-\overline{\lambda_{7/60}}=\textup{i}\frac{7\sqrt{3}}{20}$.
The vector $\mathbf{z}$ is an eigenvector of $\left(\frac{7}{60}\mathbf{M}-\mathbf{C}\right)$
with respect to the eigenvalue $\lambda$ and an eigenvector of $\left(\frac{7}{60}\mathbf{M}-\mathbf{C}\right)^{T}$
with respect to the eigenvalue $-\lambda$, so that in the preceding
statement we have $\mathbf{p}=\mathbf{q}=\mathbf{z}$. 

The most convenient way to identify one by one the terms involved
in the expression of the first Lyapunov coefficient is to use again
the properties of circulant matrices. Doing so, we find:

\begin{align*}
\mathbf{b}\left(\mathbf{z},\overline{\mathbf{z}}\right) & =-2\Real\left(\overline{\mathbf{z}}\circ\mathbf{C}\mathbf{z}\right)\\
& =-2\Real\left(\frac{7\textup{j}}{10}\overline{\mathbf{z}}\circ\mathbf{z}\right)\\
 & =\frac{7}{30}\mathbf{1},
\end{align*}

\begin{align*}
\mathbf{A}^{-1} & =\left(\mathbf{U}\left(\begin{matrix}-1 & 0 & 0\\
0 & \lambda & 0\\
0 & 0 & -\lambda
\end{matrix}\right)\overline{\mathbf{U}}^{T}\right)^{-1}\\
 & =\mathbf{U}\left(\begin{matrix}\left(-1\right)^{-1} & 0 & 0\\
0 & \left(\textup{i}\frac{7\sqrt{3}}{20}\right)^{-1} & 0\\
0 & 0 & \left(-\textup{i}\frac{7\sqrt{3}}{20}\right)^{-1}
\end{matrix}\right)\overline{\mathbf{U}}^{T}\\
 & =\mathbf{U}\left(\begin{matrix}-1 & 0 & 0\\
0 & -\textup{i}\frac{20\sqrt{3}}{21} & 0\\
0 & 0 & \textup{i}\frac{20\sqrt{3}}{21}
\end{matrix}\right)\overline{\mathbf{U}}^{T},
\end{align*}

\[
\mathbf{A}^{-1}\mathbf{b}\left(\mathbf{z},\overline{\mathbf{z}}\right)=-\frac{7}{30}\mathbf{1},
\]

\begin{align*}
\mathbf{b}\left(\mathbf{z},\mathbf{A}^{-1}\mathbf{b}\left(\mathbf{z},\overline{\mathbf{z}}\right)\right) & =\frac{7}{30}\left(\mathbf{C}\mathbf{z}+\mathbf{z}\right)\\
 & =\frac{7}{300}\left(10+7\textup{j}\right)\mathbf{z},
\end{align*}

\[
\overline{\mathbf{\mathbf{z}}}^{T}\mathbf{b}\left(\mathbf{z},\mathbf{A}^{-1}\mathbf{b}\left(\mathbf{z},\overline{\mathbf{z}}\right)\right)=\frac{7}{300}\left(10+7\textup{j}\right),
\]

\begin{align*}
\mathbf{b}\left(\mathbf{z},\mathbf{z}\right) & =-2\mathbf{z}\circ\mathbf{C}\mathbf{z}\\
& =-\frac{14\textup{j}}{10}\mathbf{z}\circ\mathbf{z}\\
 & =-\frac{14\sqrt{3}}{30}\textup{j}\overline{\mathbf{z}},
\end{align*}

\begin{align*}
\left(\textup{i}\frac{7\sqrt{3}}{10}\mathbf{I}-\mathbf{A}\right)^{-1} & =\left(\mathbf{U}\left(\begin{matrix}1+\textup{i}\frac{7\sqrt{3}}{10} & 0 & 0\\
0 & -\lambda+\textup{i}\frac{7\sqrt{3}}{10} & 0\\
0 & 0 & \lambda+\textup{i}\frac{7\sqrt{3}}{10}
\end{matrix}\right)\overline{\mathbf{U}}^{T}\right)^{-1}\\
 & =\mathbf{U}\left(\begin{matrix}\left(1+\textup{i}\frac{7\sqrt{3}}{10}\right)^{-1} & 0 & 0\\
0 & \left(-\textup{i}\frac{7\sqrt{3}}{20}+\textup{i}\frac{7\sqrt{3}}{10}\right)^{-1} & 0\\
0 & 0 & \left(\textup{i}\frac{7\sqrt{3}}{20}+\textup{i}\frac{7\sqrt{3}}{10}\right)^{-1}
\end{matrix}\right)\overline{\mathbf{U}}^{T}\\
 & =\mathbf{U}\left(\begin{matrix}\frac{100-\textup{i}70\sqrt{3}}{247} & 0 & 0\\
0 & -\textup{i}\frac{20\sqrt{3}}{21} & 0\\
0 & 0 & -\textup{i}\frac{20\sqrt{3}}{63}
\end{matrix}\right)\overline{\mathbf{U}}^{T},
\end{align*}

\begin{align*}
    \left(\textup{i}\frac{7\sqrt{3}}{10}\mathbf{I}-\mathbf{A}\right)^{-1}\mathbf{b}\left(\mathbf{z},\mathbf{z}\right) & =-\frac{14\sqrt{3}}{30}\textup{j}\left(\textup{i}\frac{7\sqrt{3}}{10}\mathbf{I}-\mathbf{A}\right)^{-1}\overline{\mathbf{z}}\\
    & =-\frac{14\sqrt{3}}{30}\textup{j}\left(-\textup{i}\frac{20\sqrt{3}}{63}\right)\overline{\mathbf{z}}\\
    & =\frac{4}{9}\textup{i}\textup{j}\overline{\mathbf{z}},
\end{align*}

\begin{align*}
    \mathbf{b}\left(\overline{\mathbf{z}},\left(\textup{i}\frac{7\sqrt{3}}{10}\mathbf{I}-\mathbf{A}\right)^{-1}\mathbf{b}\left(\mathbf{z},\mathbf{z}\right)\right) & =\frac{4}{9}\textup{i}\textup{j}\mathbf{b}\left(\overline{\mathbf{z}},\overline{\mathbf{z}}\right)\\
    & =\frac{4}{9}\textup{i}\textup{j}\overline{\mathbf{b}\left(\mathbf{z},\mathbf{z}\right)}\\
    & =-\frac{56\sqrt{3}}{270}\textup{i}\mathbf{z}\\
 & =-\textup{i}\frac{28\sqrt{3}}{135}\mathbf{z},
\end{align*}

\[
    \overline{\mathbf{z}}^{T}\mathbf{b}\left(\overline{\mathbf{z}},\left(\textup{i}\frac{7\sqrt{3}}{10}\mathbf{I}-\mathbf{A}\right)^{-1}\mathbf{b}\left(\mathbf{z},\mathbf{z}\right)\right)=-\textup{i}\frac{28\sqrt{3}}{135}
\]

Finally, the first Lyapunov coefficient of (\ref{sys:2}) is 
\begin{align*}
    l_{1}\left(0\right) & =\frac{10\sqrt{3}}{21}\Real\left(-\frac{14}{300}\left(10+7\textup{j}\right)-\textup{i}\frac{28\sqrt{3}}{135}\right)\\
    & =-\frac{10\sqrt{3}}{21}\times\frac{14}{300}\times\frac{13}{2}\\
    & =-\frac{13\sqrt{3}}{90}
\end{align*}
and, consequently, the limit cycle $\mathsf{C}_{\mu}$
is indeed locally asymptotically stable close to the bifurcation value.
\end{proof}

\subsection{Continuation of the limit cycle when the Hopf bifurcation theorem does not apply anymore}
\begin{proof}
    First, we show that, for any $\mu\geq 0$, the $\omega$-limit 
    set of (\ref{sys:2}) is, apart from $\mathbf{0}$, contained in the compact set
    \[
	\mathsf{I}=\left\{ \mathbf{v}\geq\mathbf{0}\ |\ 1\leq\frac{v_1+v_2+v_3}{3}\leq\frac{10}{3} \right\}.
    \]
    Using again the decomposition $\mathbf{v}=\alpha\mathbf{1}+\beta\mathbf{z}+\overline{\beta\mathbf{z}}$, this simply amounts to verifying that, for any $\alpha\in\left[ 0,1 \right]$, 
    \[
	\left( \mathbf{v}+\mu\mathbf{M}\mathbf{v}-\mathbf{C}\mathbf{v}\circ\mathbf{v} \right)\cdot\mathbf{1}\geq 0
    \]
    and, for any $\alpha\geq\frac{10}{3}$,
    \[
	\left( \mathbf{v}+\mu\mathbf{M}\mathbf{v}-\mathbf{C}\mathbf{v}\circ\mathbf{v} \right)\cdot\mathbf{1}\leq 0.
    \]
    Using again previous calculations, we end up with 
    \[
	\left( \mathbf{v}+\mu\mathbf{M}\mathbf{v}-\mathbf{C}\mathbf{v}\circ\mathbf{v} \right)\cdot\mathbf{1}=\alpha-\alpha^2+\frac{7}{30}\left|\beta\right|^2,
    \]
    which is obviously nonnegative if $\alpha\in\left[ 0,1 \right]$. Noticing the simple geometric
    fact that 
    \[
	\mathsf{T}_\alpha=\left\{\mathbf{v}\geq\mathbf{0}\ |\ v_1+v_2+v_3=3\alpha\right\}
    \]
    is an equilateral triangle of perimeter $9\sqrt{2}\alpha$ whose circumscribed circle is the
    boundary of the closed two-dimensional ball
    \[
	\overline{\mathsf{B}}_\alpha=\left\{\mathbf{v}\in\overline{\mathsf{B}\left( \alpha\mathbf{1},\sqrt{6}\alpha \right)}\ |\ v_1+v_2+v_3=3\alpha\right\},
    \]
    we deduce
    \begin{align*}
	\sqrt{6}\alpha & = \max_{\mathbf{v}\in\mathsf{T}_\alpha}\left|\mathbf{v}-\alpha\mathbf{1}\right| \\
	& = \max_{\mathbf{v}\in\overline{\mathsf{B}}_\alpha}\left|\mathbf{v}-\alpha\mathbf{1}\right| \\
	& = \max_{\mathbf{v}\in\overline{\mathsf{B}}_\alpha}\left|2\Real\left( \beta\mathbf{z} \right)\right| \\
	& = 2\max_{\mathbf{v}\in\overline{\mathsf{B}}_\alpha}\left(\left|\beta\right| \left|\Real\left( \textup{e}^{\textup{i}\arg\left( \beta \right)}\mathbf{z} \right)\right| \right) \\
	& = \frac{2}{\sqrt{3}}\max_{\mathbf{v}\in\overline{\mathsf{B}}_\alpha}\left|\beta\right| \max_{\theta\in\left[ 0,2\pi \right]}\sqrt{\cos\left( \theta \right)^2+\cos\left( \theta+\frac{2\pi}{3} \right)^2+\cos\left( \theta+\frac{4\pi}{3} \right)^2} \\
	& = \sqrt{2}\max_{\mathbf{v}\in\overline{\mathsf{B}}_\alpha}\left|\beta\right| \\
	& \geq \sqrt{2}\max_{\mathbf{v}\in\mathsf{T}_\alpha}\left|\beta\right|
    \end{align*}
    whence 
    $\alpha-\alpha^2+\frac{7}{30}\left|\beta\right|^2\leq\alpha-\frac{3}{10}\alpha^2$,
    which is indeed negative if $\alpha>\frac{10}{3}$ (and, having in mind that
    $10\mathbf{e}_1=\frac{10}{3}\mathbf{1}+2\Real\left( \frac{10}{3}\mathbf{z} \right)$
    is a steady state of the particular case $\mu=0$, this constant is optimal).

    Considering the dynamical system defined by (\ref{sys:2}) with initial conditions in the unstable 
    manifold of $\mathbf{1}$, we can reduce it to a two-dimensional flow whose $\omega$-limit set is
    also included in $\mathsf{I}$. Applying the Poincaré--Bendixson theorem, we deduce for any value of
    $\mu\in\left( 0,\mu_{\textup{H}} \right)$ the necessary existence of a positive limit cycle 
    $\mathsf{C}_\mu$ in $\mathsf{I}$.

    Notice that although a limit cycle that is
    locally asymptotically stable for the flow embedded in the unstable manifold of $\mathbf{1}$
    does exist, we do not, at this point, have any information on the stability of this limit cycle
    in the three-dimensional flow.
    
    Next, using the relative compactness (in the topology induced by the Hausdorff distance) 
    of any family 
    $\left( \mathsf{C}_\mu \right)_{\mu\in\left( 0,\mu_{\textup{H}}\right)}$, we can extract
    a limit point of it as $\mu\to 0$, say $\mathsf{C}$. Fixing an appropriate family
    of initial conditions, we easily derive the existence of a solution of (\ref{sys:2}) with
    $\mu=0$ whose full trajectory is contained in $\mathsf{C}$. The corresponding orbit is 
    a fixed point, a limit cycle, a heteroclinic connection or a homoclinic connection. 
    
    Since $\mathbf{1}$ does not bifurcate again at $\mu=0$, the case 
    $\mathsf{C}=\left\{ \mathbf{1} \right\}$ is discarded.

    The well-known characterization of the $\omega$-limit set of the three-component Lotka--Volterra
    competitive system corresponding to the case $\mu=0$ (see Zeeman 
    \cite[equivalence class n°27, p. 22]{Zeeman_1993}, Uno--Odani \cite{Uno_Odani_1997}, 
    May--Leonard \cite{May_Leonard_1975}, Petrovskii--Kawasaki--Takasu--Shigesada 
    \cite{Petrovskii_Kawasaki_Takasu_Shigesada}, etc.) shows then that $\mathsf{C}$ is 
    indeed a reunion of elements among $\left\{ 10\mathbf{e}_1 \right\}$, 
    $\left\{ 10\mathbf{e}_2 \right\}$, $\left\{ 10\mathbf{e}_3 \right\}$,
    $\mathsf{H}_1$, $\mathsf{H}_2$, $\mathsf{H}_3$ (where we recall that $\mathsf{H}_i$ is
    the heteroclinic orbit connecting $10\mathbf{e}_{i}$ and $10\mathbf{e}_{i+1}$). 
    In particular, the limiting system does not admit any periodic limit cycle. 

    In order to conclude, it only remains to prove that any limit cycle $\mathsf{C}_\mu$ encloses 
    $\vspan\left(\mathbf{1}\right)$, so that in the end $\mathsf{C}=\mathsf{C}_0$.
    To do so, we are going to show that the flow always crosses a plane containing the straight
    line $\vspan\left( \mathbf{1} \right)$ in the same direction, that is we are going
    to show that, for any $\mathbf{v}=\alpha\mathbf{1}+\beta\mathbf{z}+\overline{\beta\mathbf{z}}$ 
    with $\beta\neq 0$,
    \[
	\left( \mathbf{v}+\mu\mathbf{M}\mathbf{v}-\mathbf{C}\mathbf{v}\circ\mathbf{v} \right)\cdot\left( \textup{e}^{\textup{i}\frac{\pi}{2}}\beta\mathbf{z}+\overline{\textup{e}^{\textup{i}\frac{\pi}{2}}\beta\mathbf{z}} \right)
    \]
    has a constant sign.
    Using once more previous calculations, this amounts to finding the sign of
    \[
	\Real\left( -\textup{i}\overline{\beta}\left( \frac{7\sqrt{3}}{30}\overline{\textup{j}\beta^{2}}+\left(\alpha-\left(1-3\mu\right)+\frac{7}{10}\left(\frac{-1+\textup{i}\sqrt{3}}{2}\right)\alpha\right)\beta \right) \right),
    \]
    that is the sign of
    \[
	\Real\left( \textup{e}^{\textup{i}\frac{5\pi}{6}}\frac{7\sqrt{3}}{30}\overline{\beta^3}-\textup{i} \left(\alpha-\left(1-3\mu\right)+\frac{7}{10}\left(\frac{-1+\textup{i}\sqrt{3}}{2}\right)\alpha\right)\left|\beta\right|^2 \right),
    \]
    that is that of
    \begin{align*}
	\frac{7\sqrt{3}}{30}\left|\beta\right|\cos\left( \frac{5\pi}{6}-3\arg\left( \beta \right) \right)+\frac{7\sqrt{3}}{20}\alpha & =\frac{7\sqrt{3}}{10}\left( \frac{\left|\beta\right|}{3} \cos\left( \frac{5\pi}{6}-3\arg\left( \beta \right) \right)+\frac{\alpha}{2}\right) \\
	& =\frac{7\sqrt{3}}{10}\left( -\frac{\left|\beta\right|}{3} \cos\left( \frac{\pi}{6}+3\arg\left( \beta \right) \right)+\frac{\alpha}{2}\right) 
    \end{align*}
    The estimate $\left|\beta\right|\leq\sqrt{3}\alpha$ is this time not precise enough; we
    truly need to relate the modulus of $\beta$ and its argument. 

    By periodicity and invariance by rotation around the axis 
    $\vspan\left( \mathbf{1} \right)$, it suffices to consider an interval of length
    $\frac{2\pi}{3}$ for the parameter $\theta=\arg\left( \beta \right)$. For instance, we 
    take the interval $\left[ \frac{2\pi}{3},\frac{4\pi}{3} \right]$. In this interval, 
    $\mathsf{T}_\alpha$ is characterized by the inequality $v_1\geq 0$, 
    which reads $\alpha+2\left|\beta\right|\cos\theta\geq 0$.
    Consequently, the studied sign is nonnegative provided 
    \[
	3\left|\cos\theta\right|\geq\left|\cos\left( 3\theta+\frac{\pi}{6} \right)\right|\text{ for all }\theta\in\left[ \frac{2\pi}{3},\frac{4\pi}{3} \right],
    \]
    which is obviously true.

    Therefore the flow is rotating clockwise around $\vspan\left( \mathbf{1} \right)$
    if seen from $\mathbf{0}$ (consistently with \figref{limit_cycle}), and so is any periodic orbit.
    Thus any limit point $\mathsf{C}$ satisfies indeed $\mathsf{C}=\mathsf{C}_0$, whence any
    full family $\left( \mathsf{C}_\mu \right)_{\mu\in\left( 0,\mu_{\textup{H}} \right)}$
    converges as $\mu\to 0$ to $\mathsf{C}_0$.
\end{proof}

It might be tempting to use the same ideas to localize more efficiently, and maybe even 
count, the limit cycles. However, the sign of
\[
    \left( \mathbf{v}+\mu\mathbf{M}\mathbf{v}-\mathbf{C}\mathbf{v}\circ\mathbf{v} \right)\cdot\left( \beta\mathbf{z}+\overline{\beta\mathbf{z}} \right)
\]
is the same as the sign of
\[
    \frac{60}{13}\left( \mu_{\textup{H}}-\mu \right)-\left( \alpha-1 \right)-\frac{14\sqrt{3}}{39}\cos\left( 3\arg\left( \beta \right)+\frac{2\pi}{3} \right)\left|\beta\right|.
\]
Given a fixed angle $\arg\beta\in\left[ 0,2\pi \right]$, the nullcline is a straight line 
in the plane of coordinates $\left( \alpha-1,\left|\beta\right| \right)$ whose slope is 
$-\frac{14\sqrt{3}}{39}\cos\left( 3\arg\beta+\frac{2\pi}{3} \right)$. Unfortunately the sign 
of this slope varies as $\arg\beta$ varies. Hence the best, and really unsatisfying, result
we can deduce from this is that any limit cycle is in the region of the phase space
where 
\[
    \frac{60}{13}\left( \mu_{\textup{H}}-\mu \right)-\left( \alpha-1 \right)\in\left[-\frac{14\sqrt{3}}{39}\left|\beta\right|,\frac{14\sqrt{3}}{39}\left|\beta\right|\right].
\]

\section{Numerical findings}

In this section, $\mu=\frac{13}{120}\in\left(\mu_{-},\mu_{\textup{H}}\right)$ is fixed.

\subsection{The numerical scheme}
All the forthcoming plots are obtained thanks to a simple finite difference scheme, 
explicit in time and with Neumann boundary conditions on the boundary of a very large spatial
interval. It is well-known that such a spatial domain approximates correctly $\mathbb{R}$,
at least regarding spreading properties of reaction--diffusion systems and equations. 
Indeed, the forthcoming results are consistent with previously known
theoretical results (such as, for instance, the fact that initially compactly supported solutions
for (\ref{sys:1}) invade $\mathbf{0}$ at speed $2$ or the exponential decay of traveling wave solutions \cite{Girardin_2016_2,Girardin_2017}).

Source codes are run in \textit{Octave} \cite{Octave}. 

The findings seem to be robust with respect to the numerical parameters.

\subsection{The limit cycle for (\ref{sys:2})}
Although we do not know how to prove analytically the global attractivity or the uniqueness of the limit cycle 
$\mathsf{C}_\mu$, numerically it seems indeed to be true, as illustrated by
\figref{limit_cycle}.

\begin{figure}
    \resizebox{\textwidth}{!}{\input{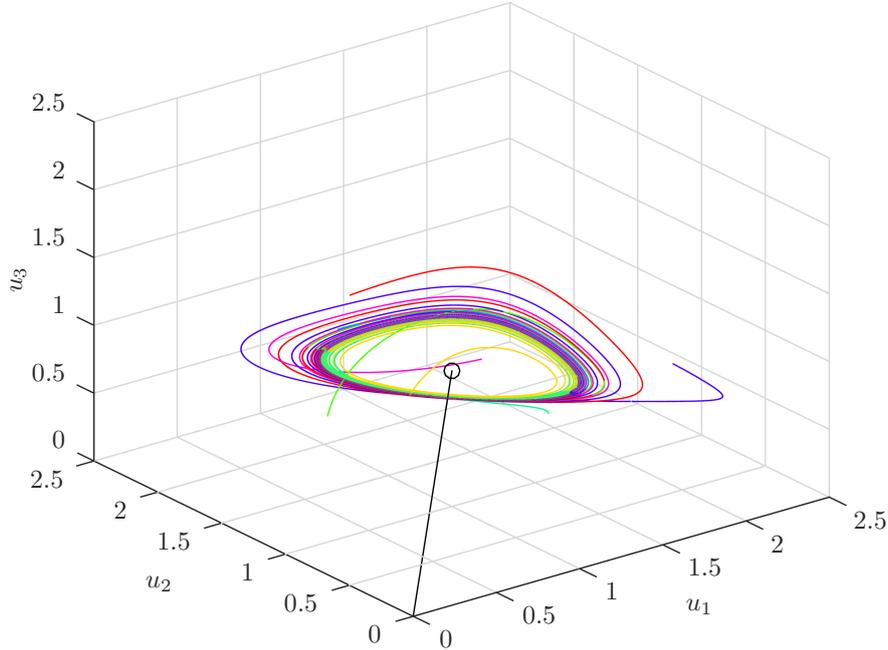}}
    \caption{\label{fig:limit_cycle} Seven trajectories of (\ref{sys:2}) with random initial conditions.}
\end{figure}

\subsection{The Cauchy problem with diffusion \label{subsec:numerics_diffusion}}
The following findings seem to be robust with respect to the initial condition $\mathbf{u}_0$, 
as soon as it is compactly supported, nonzero and not in $\vspan(\mathbf{1})$ 
(stable manifold of $\mathbf{1}$). For instance, we fix 
$\mathbf{u}_0=(1.01,1.01,0.99)^T$ in a small interval in the center of the domain
and $\mathbf{u}_0=\mathbf{0}$ elsewhere.

\begin{figure}
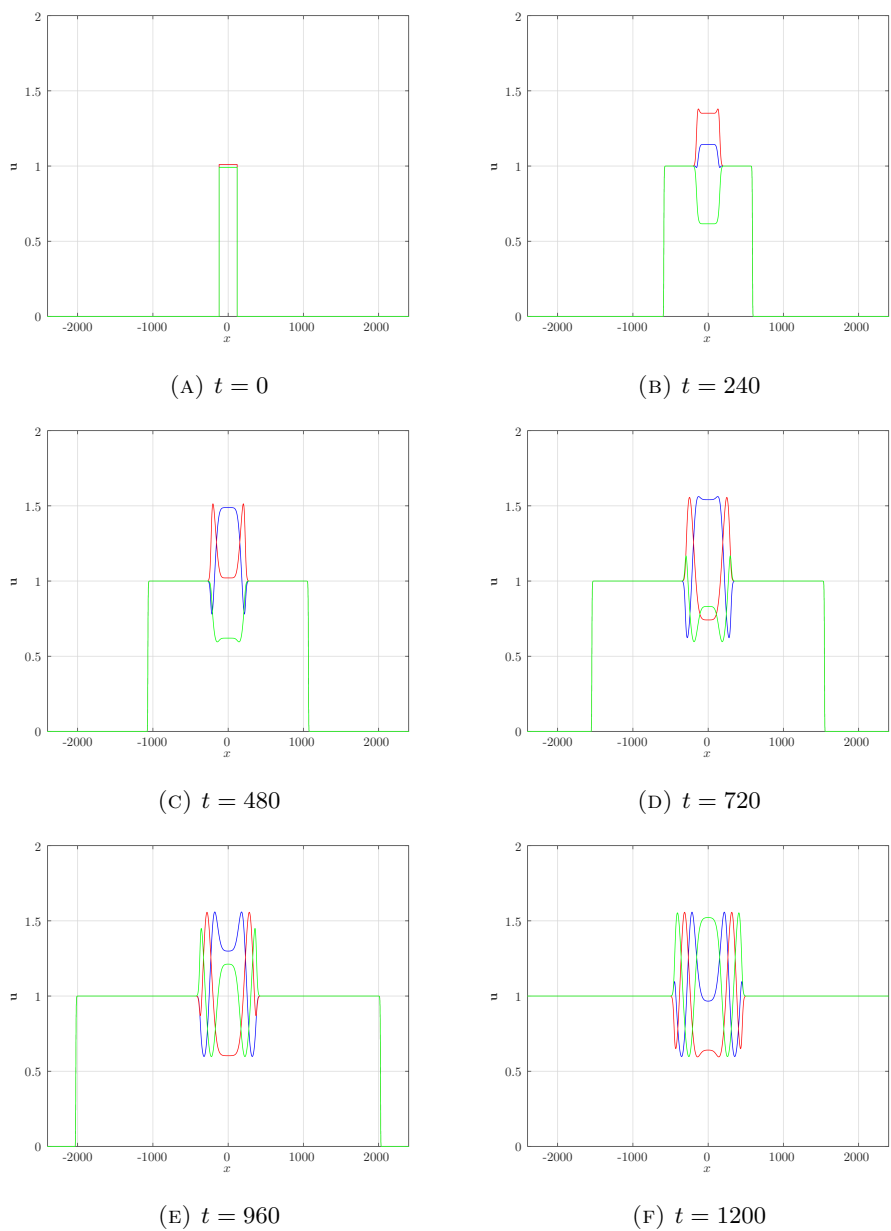

    \begin{subfigure}{.49\textwidth}
	\resizebox{\textwidth}{!}{\input{evolution_0.tex}}
	\caption{$t=0$}
    \end{subfigure}
    \begin{subfigure}{.49\linewidth}
	\resizebox{\textwidth}{!}{\input{evolution_2.tex}}
	\caption{$t=240$}
    \end{subfigure}
    \begin{subfigure}{.49\linewidth}
	\resizebox{\textwidth}{!}{\input{evolution_4.tex}}
	\caption{$t=480$}
    \end{subfigure}
    \begin{subfigure}{.49\linewidth}
	\resizebox{\textwidth}{!}{\input{evolution_6.tex}}
	\caption{$t=720$}
    \end{subfigure}
    \begin{subfigure}{.49\linewidth}
	\resizebox{\textwidth}{!}{\input{evolution_8.tex}}
	\caption{$t=960$}
    \end{subfigure}
    \begin{subfigure}{.49\linewidth}
	\resizebox{\textwidth}{!}{\input{evolution_10.tex}}
	\caption{$t=1200$}
    \end{subfigure}
    \caption{\label{fig:evolution} Snapshots of the Cauchy problem.}
\end{figure}

Once the existence of a pulsating front connecting $\mathbf{1}$ to a wave
train $(t,x)\mapsto\mathbf{p}_\gamma\left( \kappa_\gamma x - \sigma_\gamma t \right)$
is observed (see \figref{evolution}), we
use the phase space to estimate the amplitude $\gamma$ of the wave train.
In order to do so, we plot in \figref{propagation_of_oscillations} the 
trajectory of $t\mapsto\mathbf{u}(t,x)$ with $x$ appropriatly chosen 
(say, away from the initial support of the solution but within the final 
support of the wave train)
together with $\mathsf{C}$ (obtained by truncating
any trajectory of (\ref{sys:2}), see \figref{limit_cycle}).
This confirms that $\gamma$ is smaller than, but close to, $1$. As a side
note, this also confirms that the selected wave train is a stable one (in
the sense of \propref{stability_wave_trains}). 

\begin{figure}
    \resizebox{\textwidth}{!}{\input{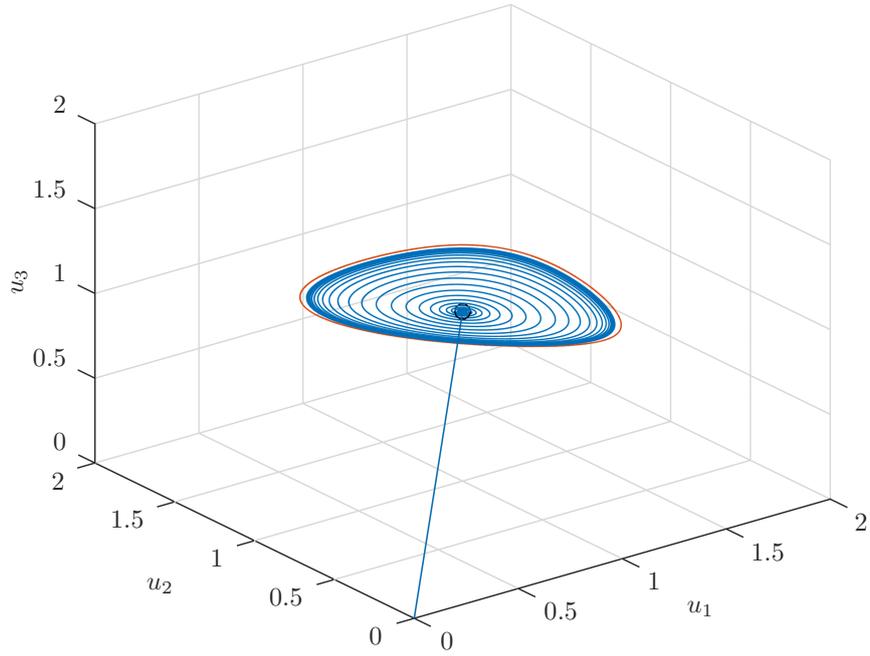}}
    \caption{\label{fig:propagation_of_oscillations} In blue, the trajectory at $x$. In red, the
    limit cycle $\mathsf{C}$.}
\end{figure}

To evaluate the speed of the pulsating front, the most convenient way is to plot an appropriate
level set. Since the three components of $\mathbf{u}$ always spread together, it is sufficient to
plot the level set of only one component, say $u_1$. Of course, the value $U$ of $u_1$ at this 
level set must satisfy $0<U<1$, so for instance we fix $U=0.9$. We obtain then \figref{level_sets}.

With \figref{level_sets}, we can verify that the invasion $\mathbf{1}\to\mathbf{0}$ occurs 
at speed $2$ and
we can evaluate graphically that the invasion $\mathbf{p}_\gamma\to\mathbf{1}$ occurs
at speed $c\simeq\frac{1}{3}$, which corresponds to the linear prediction of Sherratt
\cite{Sherratt_1998}:
\[
    c_{\textup{lin}}=2\sqrt{\Real\left(\lambda_{\frac{13}{120}}\right)}=2\sqrt{\frac{3}{120}}=\frac{1}{\sqrt{10}}\simeq 0.3162.
\]

Let us point out that Sherratt's prediction uses the parameter $\lambda_0$ of the $\lambda$-$\omega$
form of the system instead of the real part of the bifurcating eigenvalues and that
the equality is perhaps not obvious. As explained earlier, the $\lambda$-$\omega$ reduction
is not performed in the present paper, but this is in fact unnecessary as far as the speed
$c_{\textup{lin}}$ 
is concerned. Indeed, to obtain the linear part of the $\lambda$-$\omega$ reduction,
it suffices to notice that in the orthogonal basis of $\mathbb{R}^3$ 
\[
    \left( \mathbf{1},\mathbf{z}+\overline{\mathbf{z}},\textup{i}\left( \mathbf{z}-\overline{\mathbf{z}} \right) \right)
    = \left( \mathbf{1},
    \frac{1}{\sqrt{3}}\begin{pmatrix}
	2 \\ -1 \\ -1
    \end{pmatrix},
    \begin{pmatrix}
	0 \\ -1 \\ 1
    \end{pmatrix}
    \right),
\]
$\mu\mathbf{M}-\mathbf{C}$ reads
\[
    \begin{pmatrix}
	-1 & 0 & 0 \\
	0 & 3\left( \frac{7}{60}-\mu \right) & \frac{7\sqrt{3}}{20} \\
	0 & -\frac{7\sqrt{3}}{20} & 3\left( \frac{7}{60}-\mu \right)
    \end{pmatrix}.
\]
Hence the parameters $\lambda_0$ and $\omega_0$ of the $\lambda$-$\omega$ normal form
are indeed the real and imaginary parts of one of the two bifurcating
eigenvalues, namely 
\[
    \lambda_0 = 3\left( \frac{7}{60}-\mu \right)\text{ and }\omega_0 = -\frac{7\sqrt{3}}{20}.
\]

We also see on \figref{level_sets} that the intrinsic speed $c_\gamma$ of the wave train is
negative and, as expected, of large absolute value.

\begin{figure}
    \resizebox{\textwidth}{!}{\input{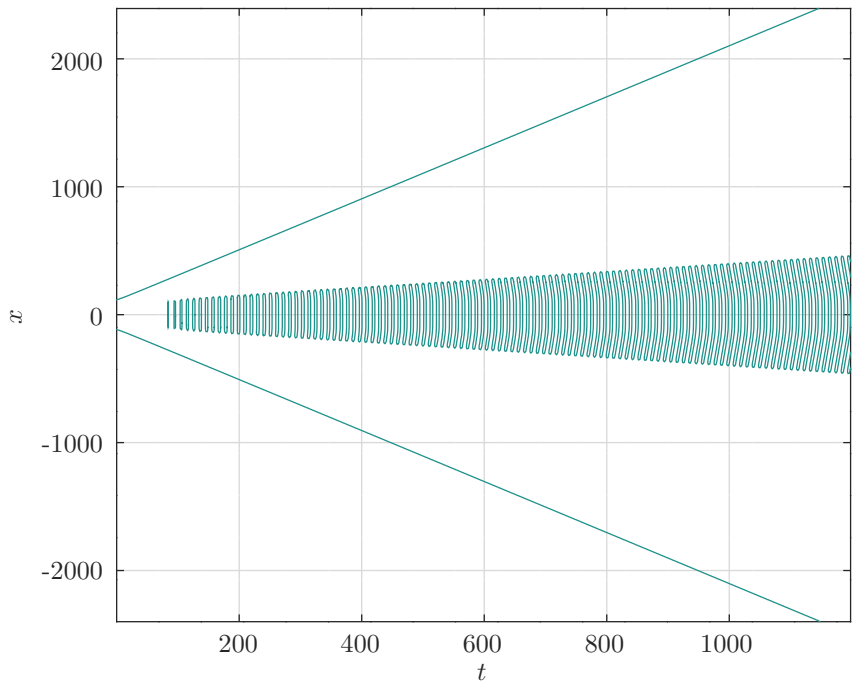}}
    \caption{\label{fig:level_sets} 0.9 level sets of $u_1$.}
\end{figure}

\section*{Acknowledgments}
The author wishes to thank two anonymous referees whose comments helped to improve the clarity of the manuscript.

\appendix
\section{Proof of \propref{stability_wave_trains}\label{app:WT}}

\begin{proof}
    As explained by Maginu \cite{Maginu_1981}, in order to establish the marginal stability in
    linearized criterion of wave trains $\mathbf{p}_\gamma$ with an amplitude $\gamma$ sufficiently
    close to 1, it suffices to prove the strong stability of the spatially homogeneous 
    limit cycle $\mathbf{p}_1$.
    Roughly speaking, the strong stability in the sense of Maginu is the linear stability 
    with respect to spatio-temporal perturbations of the form 
    $\sin\left( \omega x \right)\mathbf{u}(t)$.

    More precisely, the strong stability of the spatially homogeneous limit cycle $\mathsf{C}_\mu$ 
    for (\ref{sys:1}) is defined by Maginu \cite{Maginu_1981} as the negativity of all Floquet 
    exponents of all systems
    \begin{equation}
	\dot{\mathbf{u}}(t)=-\omega^2\mathbf{u}(t)+\mathbf{A}(t)\mathbf{u}(t)\text{ with }\omega\in\mathbb{R},
	\label{sys:strong_stability_limit_cycle}
    \end{equation}
    where $t\mapsto\mathbf{A}(t)$ is the linearization of 
    $\mathbf{v}\mapsto\mathbf{v}+\mu\mathbf{M}\mathbf{v}-\left(\mathbf{C}\mathbf{v}\right)\circ\mathbf{v}$
    evalued at $\mathbf{p}_1$ (which is the periodic profile corresponding to the limit cycle).

    Let $\mathbf{U}_\omega$ be the fundamental solution associated with 
    (\ref{sys:strong_stability_limit_cycle}), namely the solution of 
    \begin{equation*}
	\begin{cases}
	    \dot{\mathbf{U}}(t)=-\omega^2\mathbf{U}(t)+\mathbf{A}(t)\mathbf{U}(t),\\
	    \mathbf{U}(0)=\mathbf{I}.
	\end{cases}
    \end{equation*}
    It is easily verified that $t\mapsto\textup{e}^{\omega^2 t}\mathbf{U}_\omega(t)$ is
    exactly $\mathbf{U}_0$. Therefore the Floquet exponents 
    $\left( \eta_i\left( \omega \right) \right)_{i\in\{1,2,3\}}$ of
    (\ref{sys:strong_stability_limit_cycle}) satisfy exactly 
    \[
	\left( \eta_i\left( \omega \right) \right)_{i\in\{1,2,3\}}=\left( \eta_i\left( 0 \right)-\omega^2 \right)_{i\in\{1,2,3\}}.
    \]
    The negativity of the family $\left( \eta_i\left( 0 \right) \right)_{i\in\{1,2,3\}}$ leads 
    to the conclusion.
    \end{proof}

\section{A remark on the entire solutions of two-component KPP systems with 
    Lotka--Volterra competition set in a Euclidean space\label{app:ES}}

In this section we will use the terminology ``eventually cooperative'', ``eventually competitive'' 
and ``mixed type''. It refers to the trichotomy of Cantrell--Cosner--Yu 
\cite[Figure 1, Proposition 2.5]{Cantrell_Cosner_Yu_2018}.

As a preliminary, we point out a result that was just hinted in Cantrell--Cosner--Yu 
\cite{Cantrell_Cosner_Yu_2018}: in the eventually competitive, bistable case, 
we can use classical arguments (unstable manifold theorem, 
Bendixson--Dulac theorem, Poincaré--Bendixson theorem) to show that, 
exactly as in the Lotka--Volterra case, there exists a partition 
$\left(\mathsf{B}_1,\mathsf{S},\mathsf{B}_2\right)$ of $\left[0,+\infty\right)^2$ such that
each $\mathsf{B}_i$ is the basin of attraction of a stable steady state whereas 
the separatrix $\mathsf{S}$ is the basin of attraction of the nonzero unstable steady
state. The separatrix is smooth, contains $\mathbf{0}$ and is, in the competitive rectangle, 
the graph of a nondecreasing function.

\begin{prop}
    \label{prop:2_compo_convergence} Let $n\in\mathbb{N}$, 
    $\mathbf{D}$ be a diagonal $2\times 2$ matrix with positive
    diagonal entries, $\mathbf{L}$ be a $2\times 2$ essentially nonnegative and 
    irreducible matrix satisfying $\lambda_{\textup{PF}}\left( \mathbf{L} \right)>0$, $\mathbf{C}$ 
    be a $2\times 2$ positive matrix and $\mathbf{u}$ be an entire solution of 
    \[
	\partial_t \mathbf{u} - \mathbf{D}\Delta \mathbf{u} = \mathbf{L}\mathbf{u} - \mathbf{C}\mathbf{u}\circ\mathbf{u}
    \]
    satisfying
    \[
	\min_{i\in\{1,2\}}\left( \inf_{(t,x)\in\mathbb{R}\times\mathbb{R}^n}u_i(t,x) \right)>0.
    \]
    
    Then $\mathbf{u}$ is a constant steady state provided one of the following conditions holds true:
    \begin{enumerate}
	\item the system is eventually cooperative;
	\item the system is eventually competitive and monostable;
	\item the system is eventually competitive, bistable and there exists $t\in\mathbb{R}$
	    such that the image of $x\mapsto\mathbf{u}(t,x)$ does not intersect the separatrix;
	\item the system is of mixed type and $\mathbf{D}=\mathbf{I}$;
	\item the system is of mixed type and $\mathbf{u}$ is spatially periodic;
    \end{enumerate}
\end{prop}

\begin{proof}\ 
    \begin{enumerate}
	\item If the system is eventually cooperative, then necessarily $\mathbf{u}$ is valued
    in the cooperative rectangle and, by comparison with a solution that does not depend on $x$ and
    the fact that the unique steady state $\mathbf{v}^\star$ is globally attractive for the 
    diffusionless system, $\mathbf{u}$ is exactly $\mathbf{v}^\star$, which is constant.
	\item Same as before.
	\item Same as before.
	\item If the system is of mixed type and $\mathbf{D}=\mathbf{I}$, then by using 
	    the Lyapunov function $V=c_1 F_1\left( u_1 \right)+c_2 F_2\left( u_2 \right)$ of 
            \cite[Lemma 3.2]{Cantrell_Cosner_Yu_2018}, we find
            \begin{align*}
                \partial_t V - \Delta V & = -\nabla\mathbf{u}^T.\textup{D}^2 V.\nabla\mathbf{u} + \nabla V\cdot\left( \partial_t \mathbf{u}-\Delta\mathbf{u} \right)\\
                & = -\nabla\mathbf{u}^T.\textup{D}^2 V.\nabla\mathbf{u} + \nabla V\cdot\left( \mathbf{L}\mathbf{u}-\mathbf{C}\mathbf{u}\circ\mathbf{u} \right)
            \end{align*}
            and then, from the convexity of $V$, it follows again that $\mathbf{u}$ has to be 
            the unique steady state $\mathbf{v}^\star$, which is constant.
	\item Same as before except we use as Lyapunov function the integral of $V$ over a spatial
	    period.
    \end{enumerate}
\end{proof}

Recalling that the profile of a traveling wave connects, in some sense, $\mathbf{0}$ to an entire
solution of the system satisfying the positivity condition above, we deduce directly various 
sufficient conditions for the convergence of the profile. In particular, all profiles of a 
monostable system with weak mutations \cite{Griette_Raoul,Morris_Borger_Crooks} converge to 
the stable state indeed. 
This is a new step toward the resolution of a conjecture presented in an earlier work 
\cite[Conjecture 1.1]{Girardin_2017} (still, let us emphasize that the bistable case remains
largely open).

Notice that the same conditions also guarantee that the nonnegative nonzero solutions 
of the Cauchy problem converge locally uniformly to a constant steady state. In particular,
two-component systems with weak mutation rates, equal diffusion rates and a unique positive 
constant steady state satisfy this convergence property. This is of course in striking contrast 
with the three-component counter-example that is the main point of the present paper.

\section{A remark on unstable constant positive steady states of KPP systems with Lotka--Volterra competition
    \label{app:USS}}
\begin{prop}
    \label{prop:unstable_steady_state} Let $N\in\mathbb{N}$ such that $N\geq2$, 
    $\mathbf{L}$ be an $N\times N$ 
    essentially nonnegative and irreducible matrix, $\mathbf{C}$ be a $N\times N$ positive matrix
    and $\mathbf{v}\in\mathbb{R}^N$ be a positive solution of 
    $\mathbf{L}\mathbf{v}=\mathbf{C}\mathbf{v}\circ\mathbf{v}$.

    Assume that $\mathbf{v}$ is unstable, in the sense that at least one eigenvalue of the 
    linearized operator 
    $\mathbf{L}_\mathbf{v}=\mathbf{L}-\diag(\mathbf{v})\mathbf{C}-\diag(\mathbf{C}\mathbf{v})$
    has a nonnegative real part.

    Then $\mathbf{L}_\mathbf{v}$ is not essentially nonnegative.
\end{prop}
\begin{proof}
    By definition of $\mathbf{v}$, $\mathbf{L}_\mathbf{v}\mathbf{v}=-\mathbf{C}\mathbf{v}\circ\mathbf{v}$.
    This vector is obviously negative. Assuming by contradiction that 
    $\mathbf{L}_\mathbf{v}$ is essentially nonnegative, we deduce by standard
    properties of essentially nonnegative matrices (\textit{e.g.}, \cite{Horn_Johnson_1})
    that the Perron--Frobenius eigenvalue of $\mathbf{L}_\mathbf{v}$, whose real part is maximal 
    among the eigenvalues, is negative. The instability of $\mathbf{v}$ is contradicted. 
\end{proof}

With more general competition terms $\mathbf{c}(\mathbf{v})$
\cite{Girardin_2017,Girardin_2016_2}, the same proof will work provided 
$\textup{D}\mathbf{c}(\mathbf{v}).\mathbf{v}$ is nonnegative. This is a fairly general
assumption, reminiscent of the known condition for the existence of traveling waves 
\cite[Theorem 1.5]{Girardin_2016_2}.

We also point out that the same simple observation ($\mathbf{L}_\mathbf{v}\mathbf{v}$ is negative)
yields other interesting properties, for instance the existence of an eigenvalue with negative real part
in the case where $\mathbf{L}_\mathbf{v}$ is symmetric.

\bibliographystyle{plain}
\bibliography{ref}

\begin{thebibliography}{10}

\bibitem{Alfaro_Coville_2012}
Matthieu Alfaro and J{\'{e}}r{\^{o}}me Coville.
\newblock Rapid traveling waves in the nonlocal fisher equation connect two
  unstable states.
\newblock {\em Appl. Math. Lett.}, 25(12):2095--2099, 2012.

\bibitem{Alfaro_Coville_Raoul}
Matthieu Alfaro, J{\'{e}}r{\^{o}}me Coville, and Ga{\"{e}}l Raoul.
\newblock Travelling waves in a nonlocal reaction-diffusion equation as a model
  for a population structured by a space variable and a phenotypic trait.
\newblock {\em Communications in Partial Differential Equations},
  38(12):2126{--}2154, 2013.

\bibitem{Arnold_Desvill}
Anton Arnold, Laurent Desvillettes, and C{\'{e}}line Pr{\'{e}}vost.
\newblock Existence of nontrivial steady states for populations structured with
  respect to space and a continuous trait.
\newblock {\em Commun. Pure Appl. Anal.}, 11(1):83--96, 2012.

\bibitem{Benichou_Calvez}
Olivier B{\'{e}}nichou, Vincent Calvez, Nicolas Meunier, and Rapha{\"{e}}l
  Voituriez.
\newblock Front acceleration by dynamic selection in fisher population waves.
\newblock {\em Physical Review E}, 86(4):041908, 2012.

\bibitem{Berestycki_Nadin_Perthame_Ryzhik}
Henri Berestycki, Gr{\'e}goire Nadin, Beno\^{i}t Perthame, and Lenya Ryzhik.
\newblock The non-local {F}isher{--}{KPP} equation: travelling waves and steady
  states.
\newblock {\em Nonlinearity}, 22(12):2813, 2009.

\bibitem{Bouin_Calvez_2014}
Emeric Bouin and Vincent Calvez.
\newblock Travelling waves for the cane toads equation with bounded traits.
\newblock {\em Nonlinearity}, 27(9):2233--2253, 2014.

\bibitem{Bouin_Calvez_2}
Emeric Bouin, Vincent Calvez, Nicolas Meunier, Sepideh Mirrahimi, Beno{\^{i}}t
  Perthame, Gael Raoul, and Rapha{\"{e}}l Voituriez.
\newblock Invasion fronts with variable motility: phenotype selection, spatial
  sorting and wave acceleration.
\newblock {\em Comptes Rendus Mathematique}, 350(15):761--766, 2012.

\bibitem{Bouin_Henderso}
Emeric Bouin, Christopher Henderson, and Lenya Ryzhik.
\newblock Super-linear spreading in local and non-local cane toads equations.
\newblock {\em ArXiv e-prints}, 2015.

\bibitem{Cantrell_Cosner_Yu_2018}
Robert~Stephen Cantrell, Chris Cosner, and Xiao Yu.
\newblock Dynamics of populations with individual variation in dispersal on
  bounded domains.
\newblock {\em Journal of biological dynamics}, 12(1):288--317, 2018.

\bibitem{Ducrot_Giletti_Matano}
Arnaud Ducrot, Thomas Giletti, and Hiroshi Matano.
\newblock Existence and convergence to a propagating terrace in one-dimensional
  reaction-diffusion equations.
\newblock {\em Trans. Amer. Math. Soc.}, 366(10):5541--5566, 2014.

\bibitem{Octave}
John~W. Eaton, David Bateman, S{\o}ren Hauberg, and Rik Wehbring.
\newblock {\em {GNU Octave} version 5.1.0 manual: a high-level interactive
  language for numerical computations}, 2019.

\bibitem{Elliott_Cornel}
Elizabeth~C. Elliott and Stephen~J. Cornell.
\newblock Dispersal polymorphism and the speed of biological invasions.
\newblock {\em PLOS ONE}, 7(7):1--10, 07 2012.

\bibitem{Faye_Holzer_14}
Gr{\'{e}}gory Faye and Matt Holzer.
\newblock Modulated traveling fronts for a nonlocal fisher-kpp equation: a
  dynamical systems approach.
\newblock {\em J. Differential Equations}, 258(7):2257--2289, 2015.

\bibitem{Fife_1979}
Paul~C. Fife.
\newblock {\em Mathematical aspects of reacting and diffusing systems},
  volume~28 of {\em Lecture Notes in Biomathematics}.
\newblock Springer-Verlag, Berlin-New York, 1979.

\bibitem{Fraile_Sabina_1985}
Jos{\'{e}}~M. Fraile and Jos{\'{e}} Sabina.
\newblock Kinetic conditions for the existence of wave fronts in
  reaction-diffusion systems.
\newblock {\em Proc. Roy. Soc. Edinburgh Sect. A}, 103(1-2):161--177, 1986.

\bibitem{Fraile_Sabina_1989}
Jos{\'{e}}~M. Fraile and Jos{\'{e}}~C. Sabina.
\newblock General conditions for the existence of a {``}critical point-periodic
  wave front{''} connection for reaction-diffusion systems.
\newblock {\em Nonlinear Anal.}, 13(7):767--786, 1989.

\bibitem{Girardin_2017}
L\'{e}o Girardin.
\newblock Non-cooperative {Fisher{--}KPP} systems: Asymptotic behavior of
  traveling waves.
\newblock {\em Mathematical Models and Methods in Applied Sciences},
  28(06):1067--1104, 2018.

\bibitem{Girardin_2016_2}
L\'{e}o Girardin.
\newblock Non-cooperative {Fisher{--}KPP} systems: traveling waves and
  long-time behavior.
\newblock {\em Nonlinearity}, 31(1):108, 2018.

\bibitem{Griette_Raoul}
Quentin Griette and Ga{\"{e}}l Raoul.
\newblock Existence and qualitative properties of travelling waves for an
  epidemiological model with mutations.
\newblock {\em J. Differential Equations}, 260(10):7115--7151, 2016.

\bibitem{Horn_Johnson_1}
Roger~A. Horn and Charles~R. Johnson.
\newblock {\em Topics in matrix analysis}.
\newblock Cambridge University Press, Cambridge, 1991.

\bibitem{Kopell_Howard_1973}
Nancy Kopell and Louis~N. Howard.
\newblock Plane wave solutions to reaction-diffusion equations.
\newblock {\em Studies in Appl. Mat.}, 52:291--328, 1973.

\bibitem{Kuznetsov_2004}
Yuri~A. Kuznetsov.
\newblock {\em Elements of applied bifurcation theory}, volume 112 of {\em
  Applied Mathematical Sciences}.
\newblock Springer-Verlag, New York, third edition, 2004.

\bibitem{Maginu_1978}
Kenjiro Maginu.
\newblock Stability of spatially homogeneous periodic solutions of
  reaction-diffusion equations.
\newblock {\em J. Differential Equations}, 31(1):130--138, 1979.

\bibitem{Maginu_1981}
Kenjiro Maginu.
\newblock Stability of periodic travelling wave solutions with large spatial
  periods in reaction-diffusion systems.
\newblock {\em J. Differential Equations}, 39(1):73--99, 1981.

\bibitem{May_Leonard_1975}
Robert~M. May and Warren~J. Leonard.
\newblock Nonlinear aspects of competition between three species.
\newblock {\em SIAM J. Appl. Math.}, 29(2):243--253, 1975.
\newblock Special issue on mathematics and the social and biological sciences.

\bibitem{Morris_Borger_Crooks}
Aled Morris, Luca B{\"{o}}rger, and Elaine C.~M. Crooks.
\newblock Individual variability in dispersal and invasion speed.
\newblock {\em ArXiv e-prints}, dec 2016.

\bibitem{Murray_II}
James~D. Murray.
\newblock {\em Mathematical biology. II}, volume~18 of {\em Interdisciplinary
  Applied Mathematics}.
\newblock Springer-Verlag, New York, third edition, 2003.
\newblock Spatial models and biomedical applications.

\bibitem{Nadin_2009}
Gr\'{e}goire Nadin.
\newblock Traveling fronts in space-time periodic media.
\newblock {\em J. Math. Pures Appl. (9)}, 92(3):232--262, 2009.

\bibitem{Nadin_Perthame_Tang}
Gr{\'{e}}goire Nadin, Beno\^{i}t Perthame, and Min Tang.
\newblock Can a traveling wave connect two unstable states?: The case of the
  nonlocal {F}isher equation.
\newblock {\em C. R. Math. Acad. Sci. Paris}, 349(9-10):553--557, 2011.

\bibitem{Petrovskii_Kawasaki_Takasu_Shigesada}
Sergei Petrovskii, Kohkichi Kawasaki, Fugo Takasu, and Nanako Shigesada.
\newblock Diffusive waves, dynamical stabilization and spatio-temporal chaos in
  a community of three competitive species.
\newblock {\em Japan J. Indust. Appl. Math.}, 18(2):459--481, 2001.
\newblock Recent topics in mathematics moving toward science and engineering.

\bibitem{PrevostPhD}
C{\'{e}}line Pr{\'{e}}vost.
\newblock {\em Applications of partial differential equations and their
  numerical simulations of population dynamics}.
\newblock PhD thesis, PhD Thesis, University of Orleans, 2004.

\bibitem{Sherratt_1998}
Jonathan~A. Sherratt.
\newblock Invading wave fronts and their oscillatory wakes are linked by a
  modulated travelling phase resetting wave.
\newblock {\em Phys. D}, 117(1-4):145--166, 1998.

\bibitem{Smith_Sherratt}
Matthew~J. Smith and Jonathan~A. Sherratt.
\newblock The effects of unequal diffusion coefficients on periodic travelling
  waves in oscillatory reaction-diffusion systems.
\newblock {\em Phys. D}, 236(2):90--103, 2007.

\bibitem{Uno_Odani_1997}
Tamiyuki Uno and Kenzi Odani.
\newblock On a {L}otka-{V}olterra model which can be projected to a sphere.
\newblock In {\em Proceedings of the Second World Congress of Nonlinear
  Analysts, Part 3 (Athens, 1996)}, volume~30, pages 1405--1410, 1997.

\bibitem{Zeeman_1993}
Mary-Lou Zeeman.
\newblock Hopf bifurcations in competitive three-dimensional {L}otka-{V}olterra
  systems.
\newblock {\em Dynam. Stability Systems}, 8(3):189--217, 1993.

\end{thebibliography}

\end{document}